\documentclass[12pt, a4]{amsart}
\usepackage{amsgen,amsmath,amstext,amsbsy,amsopn,amsfonts,amssymb}
\usepackage{amsmath,amssymb,amsfonts,amsthm,enumerate}
\usepackage{mathrsfs}

\newtheorem{theorem}{Theorem}
\newtheorem{lemma}[theorem]{Lemma}

\newtheorem{corollary}[theorem]{Corollary}
\newtheorem{proposition}[theorem]{Proposition}
\newtheorem{obs}[theorem]{Observation} \newtheorem{defi}[theorem]{Definition}

\newtheorem{exa}[theorem]{Example}

\newtheorem{rem}[theorem]{Remark}
\newenvironment{remark}{\begin{rem}\rm}{\end{rem}}
\newtheorem{rems}[theorem]{Remarks}

\newtheorem{ack}[theorem]{Acknowlegment}

\def\n{\noindent}
\def\H{\mathcal H}

\def\K{\mathcal K}

\def\M{\mathcal M}

\def\A{\mathcal A}
\def\B{\mathcal B}

\def\O{\mathcal O}

\def\ZZ{{\mathbf Z}}
\def\CCC{{\mathbf C}}
\def\RRR{{\mathbf R}}
\def\QQ{\mathbf Q}
\def\FF{\mathbf F}

\def\RR+{{\mathbf R}^*}

\def\KK{\mathbf K}

\def\GG{\mathbf G}
\def\PP{\mathbf P}
\def\Un{\mathbf 1}

\def\Q_p{{\mathbf Q}_p}

\def\tous{\qquad\text{for all}\quad}

\def\Aut{{\rm Aut}}

\newcommand{\SL}{\operatorname{SL}}
\newcommand{\GL}{\operatorname{GL}}
\newcommand{\PGL}{\operatorname{PGL}}

\newcommand{\Rep}{\operatorname{Rep}}
\newcommand{\Char}{\operatorname{Char}}
\newcommand{\Ind}{\operatorname{Ind}}
\newcommand{\Comm}{\operatorname{Comm}}

\begin{document}

\title[Infinite characters]{Infinite characters on $GL_n(\QQ)$, 
on  $\SL_n(\ZZ),$  and  on groups acting on trees}

\address{Bachir Bekka \\ Univ Rennes \\ CNRS, IRMAR--UMR 6625\\
Campus Beaulieu\\ F-35042  Rennes Cedex\\
 France}
\email{bachir.bekka@univ-rennes1.fr}
\author{Bachir Bekka}

\thanks{The author acknowledges the support  by the ANR (French Agence Nationale de la Recherche)
through the projects Labex Lebesgue (ANR-11-LABX-0020-01) and GAMME (ANR-14-CE25-0004)}
\begin{abstract}
Answering  a question of J. Rosenberg from \cite{Rose--89}, we construct the first examples of infinite characters on $\GL_n(\KK)$ for a global field $\KK$ and  $n\geq 2.$ 
The case $n=2$ is deduced from the following more general result. Let $G$ a   non amenable countable subgroup acting on locally finite tree $X$. Assume   either  that  the stabilizer in $G$ of every vertex of $X$ is finite or that
the closure of  the image of $G$ in $\Aut(X)$ is not amenable.
We show  that  $G$  has  uncountably many infinite dimensional irreducible unitary representations $(\pi, \H)$ 
of $G$ which are  traceable, that is, such that  the $C^*$-subalgebra of $\B(\H)$ generated by $\pi(G)$ contains the algebra  of the compact operators on $\H.$ 
In the case $n\geq 3,$  we prove the existence of infinitely many characters for  $G=\GL_n(R)$, where $n\geq 3$ and $R$  is  an integral domain  such that $G$ is not amenable. In particular, the group $\SL_n(\ZZ)$ has infinitely many such characters
for $n\geq 2.$

\end{abstract}

\maketitle
\section{Introduction}
\label{S0}
Let $G$ be a  countable discrete group and $\widehat G$  the unitary dual of $G$, that is, the set of
equivalence classes of irreducible unitary representations of $G$.  
The space  $\widehat G$, equipped with
a natural Borel  structure, is  a standard Borel space exactly when  $G$ is   virtually abelian,
by results of Glimm and Thoma (see \cite{Glim--61} and \cite{Thom--68}). 
So, unless $G$ is   virtually abelian (in which case 
the representation theory of $G$ is well understood),  a description of $\widehat G$ is hopeless or useless.
There are at least two other dual objects of $G$, which seem to be more accessible  than $\widehat G$:
\begin{itemize}

\item \textbf{Thoma's dual space} $E(G)$, that is, the set  of indecomposable
positive definite central functions on $G;$ 
\item the space $\Char(G)$ of \textbf{characters} of $G,$ that is, the space of 
 lower semi-continuous semi-finite (not necessarily finite)  traces $t$ on the  \textbf{maximal $C^*$-algebra}  $C^*(G)$
 of $G$  (see Subsection~\ref{SS1}) which satisfies the following extremality condition: 
every lower semi-continuous semi-finite trace on $C^*(G)$ dominated by $t$  on the ideal of definition of $t$ is proportional to $t$.
\end{itemize}
\par
The space $\Char(G)$  parametrizes the quasi-equivalence classes of factorial  representations  of $C^*(G)$
which are \textbf{traceable} ; recall that a unitary representation $\pi$ is factorial if the von Neumann algebra $\M$ generated by $\pi(G)$ is  a factor and that a factorial representation $\pi$ is traceable if there exists a faithful normal  (not necessarily finite) trace $\tau$ on $\M$ 
and a positive element $x\in C^*(G)$  such that $0<\tau(\pi(x))<+\infty.$
If this is the case, then $t=\tau\circ \pi$ belongs to  $\Char(G)$.
Conversely, every  element of $\Char(G)$ is obtained in this way.
Traceable representations are also called \emph{normal} representations.

Two traceable factorial representations  $\pi_1$  and $\pi_2$
are quasi-equivalent if there exists an isomorphism $\Phi: \M_1 \to \M_2$ 
such that $\Phi(\pi_1(g)) = \pi_2(g)$ for all $g \in G$,  where  $\M_i$ is the factor   generated by $\pi_i(G).$ 

Observe that an \emph{irreducible} unitary representation $(\pi, \H)$ of $G$ is traceable 
 if and only if $\pi(C^*(G))$ contains the algebra of compact operators on $\H$.
 The character associated to such a representation is given by the usual trace on $\B(\H)$ 
 and so does not belong to $E(G)$ whenever $\H$ is infinite dimensional; in this case,
 the  character is said to be of type $I_\infty,$ in accordance with the type classification
 of von Neumann algebras.
 Observe also that two  irreducible traceable representations of a group $G$ are quasi-equivalent if and only if they are unitarily equivalent.

Thoma's dual space $E(G)$  is a subspace of $\Char(G)$ and classifies the quasi-equivalence classes of the factorial representations $\pi$ of $C^*(G)$ for which  the factor  $\M$ generated by $\pi(G)$ is \emph{finite}, that is, such that the trace $\tau$ on $\M$  takes only finite values (for more detail on all of this, see Chapters 6 and 17 in \cite{Dixm--C*}).
 \par

Thoma's dual space  $E(G)$ was determined for several examples of countable groups $G$,
among them  $G=GL_n(\KK)$ or $G=SL_n(\KK)$ for an infinite field $\KK$ and $n\geq 2$ (\cite{Kirillov}; see also  \cite{PeTh--16}),
and $G=SL_n(\ZZ)$ for $n\geq 3$ (\cite{Bekk--07});
a procedure  is given in \cite[Proposition 3]{Howe--77} to compute $E(G)$ when $G$ is a nilpotent  finitely generated group.

The space $\Char(G)$ has been described for some \emph{amenable} groups $G$:  
\begin{itemize}
\item  when $G$ is nilpotent, we have $E(G)=\Char(G)$  (see \cite[Theorem 2.1]{CaMo--84});  
\item the space $\Char(G)$ is determined in \cite{Guic--63}
 for the Baumslag-Solitar  group $BS(1,2)$ and in \cite{Vershik} for the infinite symmetric group; 
 \item for $G=GL_n(\KK)$ and $n\geq 2$,
it is shown   in  \cite{Rose--89}  that $E(G)=\Char(G)$  in the case where $\KK$ an algebraic  extension of a finite field.
(Observe that $GL_n(\KK)$ is amenable  if and only if  $\KK$ an algebraic extension of a finite field; see Proposition 9 in \cite{HoRo--89} or Proposition~\ref{Pro-AmenableSLn} below.)
  \end{itemize}

  J. Rosenberg asked in  \cite[Remark after Th\'eor\`eme 1]{Rose--89} whether there exists 
 an \emph{infinite} character on $G=\GL_n(\KK)$, that is, whether $\Char(G)\neq E(G)$,
  for a field $\KK$   which is not an algebraic extension of a finite field. 
 We will show below that   the answer to this question is positive, by exhibiting  as far we know
  the first examples of such characters. 
 The case where $n=2$ and $\KK$ is a global field (see below) will  be deduced from a general result concerning groups acting on trees,  which we now state.

\medskip
Recall that   a graph $X$ is locally finite if every vertex on $X$  has only finitely many neighbours. 
In this case, the group $\Aut(X)$  of automorphisms of $X,$ equipped with the topology of pointwise convergence,
is a locally compact group for which  the vertex stabilizers are compact.
Concerning the notion of weakly equivalent representations, see Chapters 3 and 18 in \cite{Dixm--C*}  (see also Section~\ref{SS1}).

\begin{theorem}
 \label{Theorem1}
Let $X$ be a  tree and $G$ a countable subgroup acting on $X$. Assume that 
\begin{itemize}
\item[(a)] either $G$ is not amenable and the stabilizer in $G$ of every vertex of $X$ is finite, or
\item[(b)] $X$ is locally finite and the closure of  the image of $G$ in $\Aut(X)$ is not amenable.
\end{itemize}
 There exists an uncountable  family  $(\pi_t)_t$ of irreducible unitary representations of $G$ with the following properties:
  $\pi_t$ is infinite dimensional, is traceable and  is not weakly equivalent to $\pi_{t'}$ for $t'\neq t.$
  \end{theorem}
  
   Recall that a global field is a finite  extension of either the field $\QQ$ of rational numbers or of the field $\FF_p(T)$ of rational
 functions in $T$ over the finite field $\FF_p$ (see Chapter III in \cite{Weil}).

 \begin{corollary}
 \label{Cor-Theorem1}
 Let $G$ be either
\begin{itemize}
\item[(i)]  $\GL_2(\KK)$ or $\SL_2(\KK)$ for  a global field $\KK$, or
\item[(ii)]  $\SL_2(\ZZ)$, or 
\item[(iii)] $F_n$, the  free  non abelian group  over $n\in \{2,\dots, +\infty\}$ generators.
\end{itemize}
There exists an uncountable family   $(\pi_t)_t$ of unitary representations of $G$ with the properties from Theorem~\ref{Theorem1}; moreover, in case $G=F_n,$ the representations $\pi_t$ are all faithful.
 \end{corollary}
 
\medskip 
 Turning to the case $n\geq 3,$  we prove a result for $G=\GL_n(R)$ or $G=\SL_n(R)$,
valid for every integral domain $R$ such  that $G$ is not amenable.
   \begin{theorem}
 \label{Theorem2}
 Let $R$ be a countable unital commutative ring which is an integral domain; in case the characteristic of $R$ is positive, assume that  the field of fractions of $R$ is not   an algebraic extension of  its prime field.
 For $n\geq 3,$ let  $G=\GL_n(R)$ or $G=\SL_n(R)$.
 There exists  an  infinite dimensional irreducible unitary representation  of $G$  which is  traceable. 
 \end{theorem}

In the  case where $R$ is a field or  the ring of integers, we can even produce infinitely  many non equivalent  representations as in  Theorem~\ref{Theorem2}.
   \begin{corollary}
 \label{Cor-Theorem2}
 \begin{itemize}
 \item[(i)] For $n\geq 3,$ let  $G=GL_n(\KK)$ for a countable field $\KK$ which is not an algebraic extension of a finite field.
 There exists an uncountable family  $(\pi_t)_t$ of    pairwise non equivalent infinite dimensional irreducible unitary representations of $G$   which are  traceable.
 Moreover, the representations $\pi_t$ all have a trivial central character, that is, the $\pi_t$'s are 
 representations of $\PGL_n(\KK).$
 \item[(ii)] Let $G=SL_n(\ZZ)$ for $n\geq 3$.   There exists  an infinite family of   pairwise non equivalent infinite dimensional irreducible unitary representations of $G$   which are  traceable.
  \end{itemize}
  \end{corollary}
   
  
  The methods of proofs of  Theorem~\ref{Theorem1} and Theorem~\ref{Theorem2} are quite different
  in nature:
  \begin{itemize}
  \item the proof of Theorem~\ref{Theorem1} is based on properties  of a remarkable family of unitary representations  
  of groups acting on trees  constructed in \cite{JuVa--84} and used  to show their $K$-theoretic amenability,   a notion  
  which originated from  \cite{Cuntz}\   in the case of free  groups;
   \item the traceable representations we construct in Theorem~\ref{Theorem2} are induced representations from suitable subgroups. The case $n\geq 4$ uses the existence of appropriate subgroups of $\GL_n(R)$  with Kazhdan's Property (T).
  \end{itemize}

 \begin{remark}
 \label{Rem-Theo2}
For a group $G$ as in Theorem~\ref{Theorem1} or  Theorem~\ref{Theorem1}, our results show that   
 the set $\Char(G)$ contains characters of type $I_\infty$.

 For, say, $G=\GL_n(\QQ)$, we do not know whether  $\Char(G)$ contains  characters of type $II_\infty$,
 that is, characters for which the corresponding factorial representation generates a factor of type $II_\infty$.

 \end{remark}
 
  This paper is organized as follows. In Section~\ref{S1}, we establish some  preliminary facts
 which are necessary to the proofs of our results. Section~\ref{S2} is devoted to the proofs  of Theorem~\ref{Theorem1}
 and Corollary~\ref{Cor-Theorem1};  Theorem~\ref{Theorem2} and Corollary~\ref{Cor-Theorem2}
 are proved in Section~\ref{S3}.

\section{Some preliminary results}.
\label{S1}
\subsection{$C^*$-algebras}
\label{SS1}

Let $G$ be a countable group. Recall that a unitary representation of $G$  is  a homomorphism
  $\pi:G\to U(\H)$ from $G$ to the unitary group  of a complex separable Hilbert space $\H.$
From now on, we will simply write \textbf{representation} of $G$ instead of  ``unitary representation of $G$''.

 Every representation  $(\pi,\H)$  of $G$ extends naturally to  a $*$-representation, denoted again by $\pi,$  of the group algebra $\CCC[G]$ by bounded operators on $\H.$
  
Recall that the maximal $C^*$-algebra $C^*(G)$ of $G$ is the completion of $\CCC[G]$  of $G$ with respect to the norm
$$
f\mapsto  \sup_{\pi\in \Rep(G)} \Vert \pi(f) \Vert,
$$
where $\Rep(G)$ denotes the set of  representations $(\pi, \H)$ of $G$ in a separable Hilbert space $\H.$ 

We can view $G$ as subset of $\CCC[G]$ and hence as a subset of $C^*(G)$.
The $C^*$-algebra  $C^*(G)$ has the following universal property: 
every  representation $(\pi, \H)$  of $G$ extends to a unique  representation (that is, $*$-homomorphism) $\pi: C^*(G)\to \B(\H)$. 
The correspondence $G\to C^*(G)$ is functorial: every   homomorphism $\varphi: G_1\to G_2$
between two countable groups $G_1$ and $G_2$  extends to a unique morphism 
$$\varphi_*: C^*(G_1)\to C^*(G_2)$$ of $C^*$-algebras. In particular, given a subgroup $H$
of a group $G,$ the injection map $i:H\to G$ extends to a morphism  $i_*:C^*(H)\to C^*(G);$ 
the map $i_*$ is injective and so $C^*(H)$ can be viewed naturally as a subalgebra of  $C^*(G)$:
indeed, this follows from  the fact that every  representation $\sigma$ of $H$ occurs as subrepresentation of
the restriction to $H$ of some  representation $\pi$ of $G$ (one may take as $\pi$ 
the induced representation $\Ind_H ^G\sigma$, as shown  below in Proposition~\ref{Prop-Induced}).

The following simple lemma will be one of our  tools in order to show that 
$\pi(C^*(G))$ contains a non-zero compact operator for a  representation $\pi$ of $G$.

Let $\A$ be a $C^*$-algebra. 
Recall that  a representation $\pi: \A\to \B(\H)$ weakly contains another 
representation  $\rho: \A\to \B(\K)$  if 
$$
\Vert \rho(a) \Vert \leq \Vert \pi(a) \Vert \tous a\in \A,
$$
or, equivalently, $\ker \pi\subset \ker \rho$ (see Chapter 3 in \cite{Dixm--C*}).
Two representations $\pi$ and $\rho$ are weakly equivalent if $\pi$ weakly contains $\rho$ and 
$\rho$ weakly contains $\pi,$ that is, if $\ker \pi=\ker \rho.$
 \begin{lemma} 
\label{lem-CompactOp}
Let $\A$ be a $C^*$-algebra and $\pi: \A\to \B(\H)$  a representation of $\A$.
Assume that $\H$ contains a non-zero finite dimensional $\pi(\A)$-invariant subspace
$\K$ and that the restriction $\pi_1$ of $\pi$ to $\K$ is not weakly contained in the restriction $\pi_0$ of $\pi$ to 
the orthogonal complement $\K^\perp$.
Then $\pi(\A)$ contains a non-zero compact operator. 
\end{lemma}
\begin{proof}
 The ideal $\ker \pi_0$ is not contained in $\ker \pi_1,$  since $\pi_1$ is not weakly contained in $\pi_0$.
 Hence, there exists $a\in \A$ with  $\pi_0(a)=0$ and $\pi_1(a)\neq 0.$
 Then  $\pi(a)=\pi_1(a)$ has a finite dimensional range and is non-zero.
 \end{proof}
 
 Knowing that a representation of $\A$ contains in its image a non-zero compact operator,
 the following lemma enables us to construct an \emph{irreducible} representation of $\A$ with the same property.

 \begin{lemma} 
\label{lem-IrredCompactOp}
Let $\A$ be a $C^*$-algebra and $\pi: \A\to \B(\H)$  a representation of $\A$  in a separable Hilbert space $\H$
Let $a\in \A$ be such that  $\pi(a)$ is a non-zero compact operator.
Then  there exists an irreducible subrepresentation $\sigma$ of $\pi$ such that $\sigma(a)$  is a compact operator
and such that  $\Vert \sigma(a)\Vert=\Vert \pi(a)\Vert.$
\end{lemma}

 \begin{proof}
 We can decompose $\pi$ as a direct integral $ \int^{\oplus}_{\Omega} \pi_{\omega} d\mu (\omega)$ of irreducible representations $\pi_\omega;$ thus, we can find
 a probabilility measure $\mu$ on a standard Borel space $\Omega$,
a  measurable field  $\omega \to \pi_{\omega}$ of irreducible  representations of $\A$ 
in a measurable field $\omega \to \H_{\omega}$
   of separable Hilbert spaces on $\Omega$,
and a Hilbert space isomorphism $U:\H \to \int^{\oplus}_{\Omega} \H_{\omega} d \mu (\omega)$
such that 
$$
U\pi(x)U^{-1} = \int^{\oplus}_{\Omega} \pi_{\omega}(x) d\mu (\omega).
\tous x\in \A 
$$
(see \cite[\S 8.5]{Dixm--C*}). Without loss of generality, we will
identify $\pi$ with  $\int^{\oplus}_{\Omega} \pi_{\omega} d\mu (\omega)$.

 Let $a\in \A$ be such that $\pi(a)$ is a non-zero compact operator.
Since $\Vert \sigma(a^*a)\Vert =\Vert \sigma(a)\Vert^2$ for every 
representation $\sigma$ of $\A,$ 
upon replacing $a$ by $a^*a,$
we can assume  that $a$ is a positive element of $\A$.
So  $\pi(a)$ is  a positive selfadjoint compact operator on $\H$ with $\pi(a)\neq 0.$

There exists an orthonormal basis $(F_n)_{n\geq 1}$ of $\H=\int^{\oplus}_{\Omega} \H_{\omega} d \mu (\omega)$
consisting of eigenvectors of $\pi(a)$,
with corresponding eigenvalues $(\lambda_n)_{n\geq 1}$, counted with multiplicities.
For every $\omega\in \Omega$  and  every $n\geq 1,$ we have 
$$
\pi_\omega(a) (F_n(\omega))=\lambda_n F_n(\omega).
\leqno{(*)}
$$

Let $n_0\geq 1$ be such that $\lambda_{n_0}=\max\{\lambda_{n}\mid n\geq 1\}$. Then  
$\Vert \pi(a)\Vert=\lambda_{n_0}.$
Set 
$$
\Omega_0=\{\omega\in \Omega\mid F_{n_0}(\omega)\neq 0.\}.
$$
Since $F_{n_0}\neq 0,$ we have $\mu(\Omega_0)>0.$
We claim that $\Omega_0$ is a finite subset of $\Omega$ 
consisting  of atoms of $\mu.$ 
Indeed, assume by contradiction that is not the case. Then there exists
an infinite  sequence $(A_k)_k$ of pairwise disjoint Borel subsets of $\Omega_0$ with $\mu(A_k)>0$.
Observe  that  $\Un_{A_k} F_{n_0}$ is a  non-zero vector in $\H$ 
and that $\langle \Un_{A_k} F_{n_0},\mid \Un_{A_l} F_{n_0}\rangle=0$ for 
every $k\neq l.$ Moreover, 
we have 
$$
\begin{aligned}
\pi(a)(\Un_{A_k} F_{n_0})&= \int^{\oplus}_{A_k} \pi_{\omega}(a) (F_{n_0}(\omega)) d\mu (\omega)\\
&= \lambda_{n_0}\int^{\oplus}_{A_k} F_{n_0}(\omega) d\mu (\omega)\\
&=\lambda_{n_0} \Un_{A_k} F_{n_0}.
\end{aligned}
$$
Since $\pi(a)$ is a compact operator and $\lambda_{n_0}\neq 0,$ this is a contradiction.

Let $\omega_0\in \Omega_0$ be such that $\mu(\{\omega_0\})>0.$ 
We claim that the linear span of $\{F_n(\omega_0)\mid n\geq 1\}$ is dense in $\H_{\omega_0}.$
Indeed, let $v\in \H_{\omega_0}$ be such that 
$$\langle v \mid F_n(\omega_0)\rangle=0 \tous n\geq 1.$$
 Let $F=\Un_{\omega_0}\otimes v\in \H$ be defined
by $F(\omega_0)=v$ and $F(\omega)=0$ for $\omega\neq \omega_0.$ 
Then $\langle F\mid F_n \rangle=0$
for all $n\geq 1.$ Hence, $F=0,$ that is, $v=0$, since $(F_n)_{n\geq 1}$ is a
basis of $\H.$ 

By $(*)$, $F_n(\omega_0)$ is an eigenvector of $\pi_{\omega_0}(a)$ 
with eigenvalue $\lambda_n$ for every $n\geq 1$ such that $F_n(\omega_0)\neq 0.$ 
Since   $\{F_n(\omega_0)\mid   n\geq 1 \}$ is a
 total subset of $\H_{\omega_0}$, it follows that there exists a basis of $\H_{\omega_0}$
 consisting of eigenvectors of $\pi_{\omega_0}(a)$.
As 
$$\lim_{n\to \infty} \lambda_n=0$$
 (in case the sequence $(\lambda_n)_{n\geq 1}$ is infinite),  it follows that $\pi_{\omega_0}(a)$ is a 
  compact operator on $\H_{\omega_0}.$
 Moreover,   we have 
 $$\Vert \pi_{\omega_0}(a)\Vert =\max\{\lambda_{n}\mid n\geq 1\}= \lambda_{n_0} =\Vert \pi(a)\Vert.$$
 
Finally,   an equivalence  between $\pi_{\omega_0}$  and a subrepresentation of $\pi$
is provided by the unitary linear map
$$\H_{\omega_0}\to \H,\,  v\mapsto \Un_{\omega_0}\otimes v.$$
 \end{proof}

\subsection{Induced representations of groups}
\label{SS2}
In the sequel, we will often  consider   group representations  which are induced representations.
 Let $G$ be a countable group, $H$ a subgroup of $G$ and $(\sigma, \K)$ a  representation of $H$.
  Recall that  the induced representation $ \Ind_H^G \sigma$ of $G$  may be realized as follows. 
 Let $\H$ be the Hilbert space of maps $f : G \to \K$ with the following properties
\begin{itemize}
\item[(i)] 
$f(hx) = \sigma(h) f(x)$ for all $x \in G, h \in H$;
\item[(ii)]
$\sum_{x \in H \backslash G} \Vert f(x) \Vert^2 < \infty.$
(Observe that
 $\Vert f(x) \Vert$ only depends on the coset of $x$ in $H\backslash G$.)
\end{itemize}
 
The induced representation
$\pi = \Ind_H^G \sigma$
is given on $\H$ by right translation: 
$$
(\pi(g) f)(x) \, = \, f(xg) 
\hskip.5cm \text{for all} \hskip.2cm 
g \in G, \hskip.1cm f \in \H
\hskip.2cm \text{and} \hskip.2cm 
x \in G.
$$

Recall that the \textbf{commensurator} of $H$ in $G$
is the subgroup, denoted by $\Comm_G (H)$, of the elements $g \in G$ such that 
$ gH g^{-1}\cap H$ is of finite index in both $H$ and $g^{-1} Hg$.

The following result  appeared in \cite{Mack--51} in the case where $\sigma$ is of dimension 1
and was extended to its present form in \cite{Klep--61} and \cite{Corw--75}.
\begin{theorem}
\label{Theo-Induced}
Let $G$ be a countable group and $H$ a subgroup of $G$
  such that  $\Comm_G (H)=H.$
\begin{itemize}
 \item [(i)] For every   \emph{finite dimensional} irreducible  representation $\sigma$ of $H,$
 the induced representation $\Ind_H^G \sigma$ is   irreducible.
 
 \item [(ii)] Let  $\sigma_1$ and $\sigma_2$ be non equivalent \emph{finite dimensional} irreducible  representations of 
 $H$. The representations $\Ind_H^G \sigma_1$ and $\Ind_H^G \sigma_2$ are non equivalent.
   
\end{itemize}
\end{theorem}
We will need to decompose the restriction to a subgroup  of an induced representation
$\Ind_H^G \sigma$ as in Theorem~\ref{Theo-Induced}.
For $g\in G,$ we denote by  $\sigma^g$
the representation of $g^{-1} Hg$ defined by $\sigma^g(x)= \sigma(gxg^{-1})$ for $x\in g^{-1} Hg.$

For the convenience of the reader, we give a short and elementary proof 
of the following  special case of the far more general result  \cite[Theorem 12.1]{Mack--52}.
 
 \begin{proposition}
\label{Prop-Induced}
Let $G$ be a countable group, $H,L$  subgroups of $G$ and $(\sigma, \K)$ a  representation of
$H.$  Let  $S$ be a system of representatives for the double coset space $H\backslash G/L$.
 The restriction  $\pi|_L$ to $L$ of the induced representation $\pi=\Ind_H^G \sigma$  is equivalent to the direct sum
$$
\bigoplus_{s\in S} \Ind_{s^{-1}H s\cap L}^L (\sigma^s|_{s^{-1}H s\cap L})
$$
\end{proposition}

\begin{proof}
Let $\H$ be the Hilbert space  of $\pi$, as  described above.
For every $s\in S,$ let $\H_s$ be the space of maps $f\in \H$ such that $f=0$ outside
the double coset $HsL.$ 
We have an orthogonal $L$-invariant decomposition 
$$
\H=\bigoplus_{s\in S} \H_s.
$$

Fix $s\in S.$ The Hilbert space $\H'_s$ of $\Ind_{s^{-1}H s\cap L}^L (\sigma^s|_{s^{-1}H s\cap L})$
consists of the maps $f: L\to \K$ such that 
\begin{itemize}
\item
$f(t x) = \sigma(sts^{-1}) f(x)$ for all $t \in s^{-1}H s\cap L, x \in L$;
\item
$\sum_{x \in s^{-1}H s\cap L\backslash L} \Vert f(x) \Vert^2 < \infty$.
\end{itemize}
Define a linear map $U: \H_s\to \H_s'$  by 
$$
Uf(x)= f(sx) \tous f\in \H_s, x\in L.
$$
Observe that, for $t\in s^{-1}H s\cap L$, $x\in L$ and $f\in \H_s,$ we have
$$
Uf(tx)=f(stx)= f((sts^{-1}) sx)= \sigma(sts^{-1}) f(sx)= \sigma^s(t) Uf(x)
$$
and that 
$$
\sum_{x \in s^{-1}H s\cap L\backslash L} \Vert Uf(x) \Vert^2 =
\sum_{x \in s^{-1}H s\cap L\backslash L} \Vert f(sx) \Vert^2=\sum_{y\in H\backslash G} \Vert f(y) \Vert^2 <\infty,
$$
so that $Uf\in \H_s'$  and $U$ is an isometry. It is easy to check that the map $U$ is invertible, with inverse given by 
$$
U^{-1}f(y)=
\begin{cases} 
  \sigma(h)f(x)   &\text{if } y=hsx\in HsL \\
   0    & \text{otherwise },
  \end{cases}
$$
for $f\in  \H_s'.$
Moreover, $U$ intertwines  the restriction of $\pi|_L$ to $\H_s$ and $\Ind_{s^{-1}H s\cap L}^L (\sigma^s|_{s^{-1}H s\cap L}):$
for $g, x\in L$ and $ f\in \H_s'$, we have 
$$
\begin{aligned}
\left(U \pi(g)U^{-1}f\right)(x)&= \left(\pi(g)U^{-1}f\right)( sx)\\
&= (U^{-1}f( sxg) \\
&= f(xg) \\
&= \left(\Ind_{s^{-1}H s\cap L}^L (\sigma^s|_{s^{-1}H s\cap L})(g) f\right)(x)
\end{aligned}
$$
\end{proof}

We will need  the following elementary lemma about induced representations containing a
finite dimensional representation.
Recall that a representation $\pi$ of a group $G$ contains another representation $\sigma$ of $G$ if 
 $\sigma$ is equivalent to a subrepresentation of $\pi.$ 
Recall also that,  if $\pi$ is finite dimensional representation of a group $G$, then 
$\pi\otimes \bar{\pi}$ contains the trivial representation $1_G$,
where $\bar\pi$ is the conjugate 
representation of $\pi$ and $\pi\otimes \rho$ denotes the (inner) tensor product 
of the  representations $\pi$ and $\rho$ (see \cite[Proposition A. 1.12]{BHV}).

\begin{proposition}
\label{Prop-Induced-SubRep}
Let $G$ be a countable group, $H$ a subgroup of $G$,
 and $\sigma$ a  representation of $H.$
Assume that  the induced representation $\Ind_H^G \sigma$ contains 
a finite dimensional representation of $G.$ Then $H$ has finite index in $G.$
\end{proposition}
\begin{proof}
By assumption, $\pi:=\Ind_H^G \sigma$ contains a  finite dimensional representation $\sigma.$
Hence, $\pi\otimes \bar{\pi}$ contains $1_G.$ 
On the other hand, 
$$\pi\otimes \bar{\pi}= (\Ind_H^G \sigma)\otimes \bar{\pi}$$
is equivalent to $\Ind_H^G (\rho)$, where $\rho= \sigma \otimes (\bar{\pi}|_H)$;
see \cite[Proposition E. 2.5]{BHV}. So, there exists  a non-zero map $f : G \to \K$ 
in the Hilbert space of $\Ind_H^G (\rho)$ which is $G$-invariant, that is,
such that $f(xg)=f(x)$ for all $g,x \in G$. This implies that the 
$L^2$-function $x\mapsto \Vert f(x) \Vert^2$ is constant on $H \backslash G.$ 
This is only possible if $H \backslash G$ is finite. 

\end{proof}

\subsection{Amenability }
\label{SS3}
Let $\GG$ be a topological group and $UCB(\GG)$ the  Banach space  of the 
 left uniformly continuous bounded functions on $\GG,$ 
 equipped with the uniform norm. Recall  that $\GG$ is amenable
 if there exists a $\GG$-invariant mean on $UCB(\GG)$ (see Appendix G in \cite{BHV}).

The following proposition characterizes the integral domains $R$ for which $\GL_n(R)$ or $\SL_n(R)$ is amenable;
the proof is an easy extension of the proof given in Proposition 9 in \cite{HoRo--89} for the case where $R$ is a field.

\begin{proposition}
\label{Pro-AmenableSLn}
Let $R$ be a countable unital commutative ring which is an integral domain.
Let $\KK$ be the field of fractions of $R$
and $G=\GL_n(R)$ or $G=\SL_n(R)$ for an integer $n\geq 2$.
The following properties are equivalent:
\begin{itemize}
\item[(i)]   $G$ is not amenable.
\item[(ii)] $\KK$  is not an  algebraic extension of a finite field.
\item[(iii)]  $R$ contains $\ZZ$ if the characteristic of $\KK$ is $0$
or the polynomial ring $\FF_p[T]$  if the characteristic of $\KK$ is $p>0.$
\end{itemize}
\end{proposition}

\begin{proof}
Assume  that $\KK$ is an algebraic extension of a finite field $\FF_q$.
Then $\KK = \bigcup_{m} \KK_m$
for an increasing family of finite extensions $\KK_m$ of $\FF_q$; 
hence, $\GL_n(\KK) = \bigcup_m \GL_n(\KK_m)$ is the
inductive limit of the finite and hence amenable groups $\GL_n(\KK_m)$; 
it follows that $\GL_n(\KK)$ is amenable and therefore $\GL_n(R)$ and $\SL_n(R)$ 
are amenable. This shows that (i) implies (ii).
\par

Assume that (ii) holds.
If the characteristic of $\KK$ is $0$, then $\KK$ contains $\QQ$ and hence $R$ contains
$\ZZ.$  So, we can assume that the characteristic of $\KK$ is $p>0.$ We claim 
 that $R$ contains  an element which is not algebraic over the prime field $\FF_p.$
Indeed, otherwise, every element in $R$ is algebraic over $\FF_p.$ 
As the set of  elements in   $\KK$  which are algebraic over  $\FF_p$ is a field,
 it would follow that  the field fraction field $\KK$ is algebraic over  $\FF_p.$ 
 This  contradiction shows that (ii) implies (iii).

Assume that (iii) holds.
Then $\SL_n(R)$ contains a copy of $\SL_2(\ZZ)$ or a copy of $\SL_2(\FF_p[T])$.
It  is well-known that both $\SL_2(\ZZ)$ and $\SL_2(\FF_p[T])$  contain a subgroup which is isomorphic
to the free group on two generators.
Therefore, $G$ is not amenable and so  (iii) implies (i).
\end{proof}

Let   $\GG$ be a  \emph{locally compact} group, with Haar measure $m.$
Recall that  the amenability of $\GG$
is characterized by the Hulanicki-Reiter theorem (see \cite[Theorem G.3.2]{BHV}):
$\GG$ is amenable if and only if the regular representation $(\lambda_{\GG}, L^2(\GG, m))$ 
weakly contains the trivial representation $1_{\GG}$, where $m$ is Haar measure on $\GG$;
when this is the  case, $\lambda_{\GG}$ weakly contains every  representation of $\GG$

The following result shows the amenability of  $\GG$ can be detected by the restriction 
of $\lambda_{\GG}$ to a dense subgroup; for a more general result, see 
\cite[Proposition~1]{GuivAsym} or \cite[Theorem 5.5]{Bekk--16}.

\begin{proposition}
\label{Pro-RestRegRep}
Let $\GG$ be a locally compact group and $G$ a countable dense subgroup of $\GG.$ 
Assume that the restriction to $G$  of the regular representation 
$\lambda_{\GG}$ of $\GG$  weakly contains the trivial representation $1_{G}$. 
Then $\GG$ is amenable.
\end{proposition}
\begin{proof}
By assumption, there exists a sequence $(f_n)_n$ in $ L^2(\GG, m)$ with $\Vert f_n\Vert=1$ such that
$$
\lim_{n} \Vert  \lambda_{\GG}(g) f_n-f_n\Vert =0 \tous g\in G.
$$
Then, since $\left| |f_n(g^{-1}x)|- |f_n(x)|\right|\leq  \left|f_n(g^{-1}x)-f_n(x)\right|$ for $g,x\in \GG,$ 
we have
$$
\lim_{n} \Vert  \lambda_{\GG}(g) |f_n|-|f_n|\Vert =0 \tous g\in G.\leqno{(*)}
$$
Set $\varphi_n:= \sqrt{|f_n|}.$ Then $\varphi_n\geq 0$ and $\int_{\GG} \varphi_n dm= 1.$ 
Every $\varphi_n$ defines a mean $M_n:f\mapsto \int_{\GG} f\varphi_n dm$ 
on $UCB(\GG)$.
 Let $M$ be  a limit of  $(M_n)_n$ for  the weak-*-topology on the dual space of $UBC(\GG).$
 It follows from $(*)$ that $M$ is invariant under $G.$ Since,
 for every $f \in UCB(\GG),$   the map 
 $$\GG\to UCB(\GG), \quad g\mapsto _gf$$
 is continuous (where $_gf$ denotes left translation by $g\in \GG$),
 it follows that $M$ is   invariant under $\GG.$
 Hence, $\GG$ is amenable.
 
 \subsection{Special linear groups over a subring of a field}
 \label{SS4}
We will  use  the following elementary lemma
about subgroups of $\SL_n(\KK)$ which stabilize a line in $\KK^n.$
\begin{lemma}
\label{Lem-InfiniteIndexSubring}
  For an infinite field $\KK$ and $n\geq 2,$ let $L$ be a subgroup of $\SL_n(\KK)$ which stabilizes a line
 $\ell$ in $\KK^n.$  Then $L\cap \SL_n(R)$  has infinite index in  $\SL_n(R)$
 for every infinite unital subring $R$ of $\KK.$
 \end{lemma}
 
 \begin{proof}
 Let $\{v_1, \dots, v_{n}\}$ be a basis of $\KK^{n}$ with $\ell= \KK{v_1}.$
Fix $i,j\in \{1,\dots, n\}$ with $i\neq j$ and, for $\lambda\in \KK,$ let 
$E_{ij}(\lambda)$ be the corresponding elementary matrix 
in $\SL_{n}(\KK),$ that is,
$$E_{ij}(\lambda)=  I_{n} + \lambda \Delta_{ij},$$
 where $\Delta_{ij}$ denotes the matrix with $1$ at the position $(i,j)$ and $0$ otherwise.

For every $l=1,\dots, n,$ let $\varphi_l:\KK\to \KK$  be  defined by 
$$
E_{ij}(\lambda)(v_1)= \sum_{i=1}^{n} \varphi_l(\lambda) v_i \quad \text{for} \quad \lambda\in \KK.
$$
Every  $\varphi_l$ is a polynomial function (in fact, an affine function)
on $\KK$ and, for  $l=2, \dots, n,$  we have 
$\varphi_l(\lambda)=0$ for every $\lambda\in \KK$ such that $E_{ij}(\lambda)\in L.$ 

Assume, by contradiction, that $L\cap \SL_n(R)$ has finite index in $\SL_{n}(R)$
for an infinite subring $R$ of $\KK.$ 
Then the subgroup 
  $$L_{i,j}(R):=L\cap \{E_{ij}(\lambda)\mid \lambda\in R\}$$ has finite index in 
the subgroup $\{E_{ij}(\lambda)\mid \lambda\in R\}$ of $SL_{n}(R).$
 In particular,  $L_{i,j}(R)$ is infinite. It follows that 
 $\varphi_l$ has infinitely many roots in $\KK$ and hence that $\varphi_l=0$, for every $l=2, \dots, n.$ 
 Therefore, every elementary 
matrix $E_{ij}(\lambda)$ fixes the line $\ell$, for  
 $i,j\in \{1,\dots, n\}$ and $\lambda\in \KK$.
Since $\SL_{n}(\KK)$ is generated by elementary matrices, 
it follows that every matrix in $\SL_{n}(\KK)$ fixes the line $\ell$; this of course is impossible.

 \end{proof}
\section{Proofs of Theorem~\ref{Theorem1} and Corollary~\ref{Cor-Theorem1}} 
\label{S2}

\subsection{Proof of Theorem~\ref{Theorem1}} 
\label{S2-SS1}
Let $X$ be a  tree, with $X^0$ the set of vertices   and $X^1$ the set of edges of $X.$
Let $\GG$ be  a locally compact group acting on $X$.

Julg and Valette constructed in \cite{JuVa--84} (see also \cite{Szwa--91} and \cite{Julg--15})
a remarkable family of  representations $(\pi_t)_{t\in [0, 1]}$ of $\GG$, all defined  on $\ell^2(X^0)$, with the following properties:
\begin{itemize}
\item[(i)]  $\pi_0$ is the natural  representation of $\GG$ on $\ell^2(X^0)$ and $\pi_1$ is equivalent to
$1_\GG\oplus \rho_1,$ where $\rho_1$ is the natural  representation of $\GG$ on $\ell^2(X^1);$ 
\item[(ii)]  for every $t\in [0, 1],$ there exists a bounded operator $T_t$ on $\ell^2(X^0)$
with inverse $T_t^{-1}$ defined on the subspace of functions of $X^0$ with finite support 
such that $\pi_t(g):=T_t^{-1} \pi_0(g) T_t$ extends to a unitary operator on $\ell^2(X^0)$
for every $g\in \GG$; so, a unitary representation $g\mapsto \pi_t(g)$  of $\GG$ is defined on  $\ell^2(X^0)$;
\item[(iii)]  $\pi_t(g)-\pi_0(g)$ is a finite-rank operator on $\ell^2(X^0),$ for every $t\in[0,1]$ and $g\in \GG;$
  \item[(iv)] we have 
  $$\langle \pi_t(g) T_t^{-1}\delta_x\mid T_t^{-1}\delta_y\rangle= t^{d(gx,y)},$$
   for every $t\in (0,1)$, $g\in \GG$ and $x,y\in X^0,$ where $d$ denotes the natural distance on $X^0;$
 \item[(v)] the map 
 $$[0,1]\to \RRR^+, \ t\mapsto  \Vert\pi_t(g)-\pi_0(g)\Vert $$  is continuous   for every $g\in \GG.$ 
\end{itemize}
(Our representation $\pi_t$ is  $g\mapsto U_{\lambda} \rho_\lambda(g) U_{t}^{-1}$
with $\lambda=-\log t,$  for the representation $\rho_\lambda$ and the operator $U_\lambda$ appearing in  \S 2 of  \cite{JuVa--84}.)

Let  $G$ be a countable group acting on $X$. Assume that
\begin{itemize}
\item[(a)] either $G$ is not amenable and the stabilizer in $G$ of every vertex of $X$ is finite or
\item[(b)] $X$ is locally finite and the closure of  the image of $G$ in $\Aut(X)$ is not amenable.
\end{itemize}
Set $\GG=G$ in case (a)  and  let $\GG$ be the closure of $G$ in $\Aut(X)$ in case (b).
 Let $(\pi_t)_{t\in [0, 1]}$ be the family of  representations of $\GG$ as above.

\vskip.2cm
$\bullet$ {\it First step.} For every $a\in C^*(G)$ and every $t\in[ 0,1],$ the operator $ \pi_t(a)-\pi_0(a)$ 
is compact and the map
 $$[0,1]\to \RRR^+, \ t\mapsto  \Vert\pi_t(a)-\pi_0(a)\Vert $$  is continuous.
  
 Indeed, this follows from  Properties (iii) and (v) of the family $(\pi_t)_{t}$ and from the fact that 
 $\CCC[G]$ is dense in $C^*(G).$
 
\vskip.2cm
$\bullet$ {\it Second step.} The restriction $\pi_0|_G$ of  $\pi_0$ to $G$  
does not weakly contain the trivial representation  $1_G.$

Indeed, the  representation $\pi_0$ of $\GG$  is  equivalent to the direct sum  $\oplus_{s\in T}  \lambda_{\GG/ \GG_s},$ where $S$ is a system
of representatives for the $\GG$-orbits in $X^0$
and $\GG_s$ is the stabilizer in $\GG$ of $s\in S$. 
Since $\GG_s$ is compact (and even finite in case (a)) and hence
amenable, $\lambda_{\GG/ \GG_s}=\Ind_{\GG_s} ^\GG 1_{\GG_s}$ is weakly contained in 
the regular representation $\lambda_{\GG}$ of $\GG$ and so  $\pi_0$
is weakly contained in $\lambda_{\GG}.$ Hence, $\pi_0$ does not weakly contain the trivial representation  $1_G$ 
in case (a). In case (b),  the claim follows from 
Proposition~\ref{Pro-RestRegRep}, since  $\GG$ is not amenable and $G$ is dense in $\GG.$ 

\vskip.2cm
$\bullet$ {\it Third step.} There exists an element 
$a\in C^*(G)$ and $0\leq t_0<1$ with the following properties:
$\pi_{t_0}(a)=0$,  $\pi_t(a)$ is a  non zero compact operator
for every $t\in(t_0, 1]$,  and the map 
$$[t_0,1] \to \RRR+, \ t\mapsto \Vert \pi_t(a)\Vert$$
is continuous.

Indeed,  by the second step, there exists $a\in C^*(G)$ such that $\pi_0(a)=0$
and $1_G(a)\neq 0$. Therefore, $\pi_1(a)\neq 0$ and  $ \pi_t(a)=\pi_t(a)-\pi_0(a)$
for every $t\in[0,1]$ and  so the claim follows from the first step.

\vskip.2cm
$\bullet$ {\it Fourth step.} Let $a\in C^*(G)$ and  $0\leq t_0<1$ be as in the third step.
 There exists an irreducible infinite dimensional subrepresentation
$\sigma_t$ of $\pi_t$ such that  $\sigma_t(a)$ is  a compact operator
and such that  $\Vert \sigma_t(a)\Vert=\Vert \pi_t(a)\Vert$ for every $t\in (t_0,1).$

Indeed, it  follows from the third step and Lemma~\ref{lem-IrredCompactOp}
that $\pi_t|_G$  contains an irreducible subrepresentation $\sigma_t$
 such that  $\sigma_t(a)$ is  a compact operator
with $\Vert \sigma_t(a)\Vert=\Vert \pi_t(a)\Vert$.
 It remains to show that $\sigma_t$ is infinite dimensional
 for every $t\in(t_0,1).$
 
 Assume, by contradiction, $\sigma_t$ is finite dimensional for some $t\in(t_0,1).$
 Since $G$ is dense in $\GG,$ the closed  subspace $\K_t$ of $\ell^2(X^0)$
 defining $\sigma_t$ is invariant under $\GG$ and so $\sigma_t$ is the restriction to $G$
of a dimensional subrepresentation of $\pi_t$, again denoted by  $\sigma_t$.
On the one hand, $\GG$ acts properly on $X^0$, since  the stabilizers of vertices 
are compact (and even finite in case (a)). So, we have 
$$\lim_{g\to +\infty: g\in \GG} d(gx,x)=0 \tous x\in X^0.$$  It follows 
from Property (iv) of the family $(\pi_t)_{t}$ that 
$\pi_t$ (and hence $\sigma_t$) is a $C_0$-representation, that is, 
$$
\lim_{g\to +\infty: g\in \GG}\langle \pi_t(g)v\mid w\rangle= 0
$$
for every $v,w\in \ell^2(X^0)$.
On the other hand, since $\sigma_t$ is  finite dimensional,
$\sigma_t\otimes\overline{\sigma_t}$ contains $1_{\GG}$.
As $\GG$ is not compact, this is a contradiction to the fact that 
$\sigma_t$ is a $C_0$-representation.

\vskip.1cm

\vskip.2cm
$\bullet$ {\it Fifth step.} There exists uncountably many real numbers
$t\in(t_0,1)$ 
such that the  subrepresentations
$\sigma_t$ of $\pi_t|_G$ as in the fourth step are pairwise non weakly equivalent.

 Indeed, by the third step, the function $f:t\mapsto \Vert \pi_t(a)\Vert$
 is continuous on $[t_0,1]$, with $f(t_0)=0$ and $f(1)>0.$
 So, the range of $f$ contains a whole interval. 
 Let  $t,s\in (t_0,1)$ be such that  $f(t)\neq f(s).$ Then  
 $$\Vert \sigma_t(a)\Vert=\Vert \pi_t(a)\Vert =f(t)\neq f(s) =  \Vert \pi_s(a)\Vert=\Vert \sigma_s(a)\Vert,$$
 and so $\sigma_t$ and $\sigma_s$ are not weakly equivalent.

\subsection{Proof of Corollary~\ref{Cor-Theorem1}}
\label{Section-Proofs-Cor-Theo1-2}
The following remarks show how Corollary~\ref{Cor-Theorem1}
follows from Theorem~\ref{Theorem1}. 

\medskip 
\n
(i) Let $\KK$ be  global field $\KK$.
Choose a non trivial  discrete valuation  $v:\KK^*\to \ZZ$.
The completion  of $\KK$ at $v$ is a non archimedean local field $\KK_v$. 
The  tree  $X_v$ associated to $v$  (see Chapter II in \cite{Serr--80}) is a locally finite regular graph.
The group $G=\GL_2(\KK)$  acts as a group of automorphisms of $X_v,$
with vertex stabilizers conjugate to $\GL_2(\O_v\cap \KK),$ where $\O_v$ is the compact subring of the integers  in $\KK_v.$
The closure of the image of $G$ in $\Aut(X_v)$ coincides with $\PGL_2(\KK_v)$
and is therefore non amenable. A similar remark applies to $G=\SL_2(\KK).$

\medskip 
\n
(ii)  As is well-known, the group  $G=\SL_2(\ZZ)$ is an amalgamated product $\ZZ/4\ZZ\ast_{\ZZ/2\ZZ}\ZZ/6\ZZ.$
It follows that $G$ acts on a tree with vertices of valence 2 or 3 
with vertex stabilizers of order $4$ or $6$ (see Chapter I, Examples 4.2. in \cite{Serr--80})

\medskip 
\noindent
(iii) The  free  non abelian group $F_2$ acts freely on its Cayley graph $X,$ which is 
a $4$-regular tree. It follows that $F_n$ acts freely on $X$ for  every $n\in \{2,\dots, +\infty\}.$ 
Observe that, in this case, the representations $\pi_t$ and $\sigma_t$ as in the proof of Theorem~\ref{Theorem1}
are faithful for $t\neq 1$  (since there are even $C_0$-representations).
\end{proof}

\section{Proofs of Theorem~\ref{Theorem2} and Corollary~\ref{Cor-Theorem2}}
\label{S3}
\subsection{Proof of Theorem~\ref{Theorem2}}
\label{SS-Proofs-Theo1}
Let $R$ be a countable unital commutative ring which is an integral domain and
 $\KK$ its  field of fractions.
 In case the characteristic of $\KK$ is positive, assume that $\KK$ is not  an algebraic extension of  its prime field.

Let $n\geq 3$ and $G= \GL_n(R)$. 
We consider the   natural action 
of $G$ on the projective space $\PP(\KK^n).$
Let $\ell_0=\KK e_1\in \PP(\KK^n)$ be the line defined by the first unit vector $e_1$ in $\KK^n.$
The stabilizer of $\ell_0$ in $G$ is
$$H= \begin{pmatrix} R^\times & R^{n-1}\\ 0 & GL_{n-1}(R) \end{pmatrix}.$$
Let $\sigma$ be a finite dimensional  representation of $H$
and $\pi:= \Ind_H ^G \sigma$. 
We claim that $\pi$ is irreducible and that $\pi(C^*(G))$ contains a non zero compact operator.
For the proof of this claim, we have to treat separately the cases $n=3$ and $n\geq 4.$

\subsubsection{\textbf{Case $n=3$}}
$\bullet$ {\it First  step.} We claim that $gHg^{-1}\cap H$ is amenable,  for every $g\in G\setminus H.$ 

Indeed, let $g\in G\setminus H.$
 Then $\ell_0$ and $g\ell_0$  are distinct lines in $\KK^n$  and are both stabilized by $gHg^{-1}\cap H$.
 Hence,  $gHg^{-1}\cap H$  is isomorphic to a subgroup of the solvable group
$$ \begin{pmatrix} \KK^* & 0& \KK \\ 0& \KK^*& \KK \\0 & 0&\KK^* \end{pmatrix}$$
and is therefore amenable.

\vskip.2cm
$\bullet$ {\it Second  step.} We claim that the representation $\pi$ is irreducible.

Indeed, in view of  Theorem~\ref{Theo-Induced}, we have to show that 
$\Comm_G(H)=H.$  Let $g\in G\setminus H.$
On the one hand, $gHg^{-1}\cap H$  is amenable, by the first step. 
 On the other hand,  $H$ is non amenable, by Proposition~\ref{Pro-AmenableSLn}. 
 This  implies that $gHg^{-1}\cap H$ is not of finite index in $H$ and so  $g$ is not in
the commensurator of $H$ in $G.$ 

\vskip.2cm
$\bullet$ {\it Third step.}  We claim that the $C^*$-algebra $\pi(C^*(G))$  contains a non-zero compact operator. 

Indeed, let $S$ be a system of representatives for the double cosets space $H\backslash G/H$
with $e\in S.$ 
 By Proposition~\ref{Prop-Induced}, the restriction  $\pi|_H$ of $\pi$  to $H$  is equivalent to the direct sum
$$
\bigoplus_{s\in S} \Ind_{s^{-1}H s\cap H}^H (\sigma^s|_{s^{-1}H s\cap H}) = \sigma\oplus \bigoplus_{s\in S\setminus\{e\}} \Ind_{s^{-1}H s\cap H}^H (\sigma^s|_{s^{-1}H s\cap H})
$$
Let $s\in S\setminus\{e\}.$ By the first step, $s^{-1}H s\cap H$ is amenable
and hence $\sigma^s|_{s^{-1}H s\cap H} $ is weakly contained in the regular representation $\lambda_{s^{-1}H s\cap H}$ of ${s^{-1}H s\cap H}$, by the Hulanicki-Reiter theorem. By continuity of induction (see \cite[Theorem F.3.5]{BHV}), it follows that 
$\Ind_{s^{-1}H s\cap H}^H (\sigma^s|_{s^{-1}H s\cap H})$ is weakly contained in
 the regular representation $\lambda_H$  of $H$.  Therefore, 
$$\pi_0:=\bigoplus_{s\in S\setminus\{e\}} \Ind_{s^{-1}H s\cap H}^H (\sigma^s|_{s^{-1}H s\cap H})$$
is weakly contained in $\lambda_H$. It follows that  
$\pi_0$ does not weakly contain $\sigma;$ indeed, assume by contradiction that 
$\sigma$ is weakly contained in $\pi_0.$ Then  
$\lambda_H\otimes \overline{\lambda_H},$ which is a multiple of $ \lambda_H,$
 weakly contains $\sigma\otimes \overline{\sigma}$. However, since  $\sigma$ is finite dimensional,
 $\sigma\otimes \overline{\sigma}$ contains $1_H$. Hence, $1_H$  is weakly contained in $ \lambda_H$
 and this is a contradiction to the non amenability of $H.$ 
 
 It follows from Lemma~\ref{lem-CompactOp} that $\pi(C^* (H))$ contains a non-zero compact 
 operator. Since $C^* (H)$ can be viewed a subalgebra of $C^* (G)$, the claim is proved
 for $G=\GL_3(R)$. 
 
\subsubsection{\textbf{Case $n\geq 4$}}

For every unital  subring $R'$ of $R,$ set
$$
L(R'):=\begin{pmatrix} 1& 0\\ 0 & \SL_{n-1}(R') \end{pmatrix},
$$
which is a subgroup of $H$  isomorphic to $SL_{n-1}(R').$

\vskip.2cm
$\bullet$ {\it First step.} Let $g_0\in G\setminus H$ and $R'$ an infinite unital subring of $R.$
We claim that $g_0Hg_0^{-1}\cap L(R')$ has infinite index in $L(R')$.

Indeed, the group $L:=g_0Hg_0^{-1}\cap L(R')$ stabilizes the two lines $\ell_0$ and $g_0\ell_0.$ Let  
 $V$ be the linear span of the $n-1$ unit vectors $e_2,\dots, e_{n}$. Denote 
 by $\ell$ the projection on $V$ of the line $g_0\ell_0$, parallel to $\ell_0.$ 
 As $g\ell_0\neq \ell_0$, we have $\ell\neq\{0\}$. Moreover, $L$
 stabilizes $\ell$, since $L$ stabilizes $\ell_0$ and $V$.
 So,  identifying $L(R')$ with the  group $ \SL_{n-1}(R'),$ 
 we see can view $L$ as a subgroup of   $ \SL_{n-1}(\KK)$
 which stabilizes a line in $\KK^{n-1}.$ 
 Lemma~\ref{Lem-InfiniteIndexSubring} shows that $L$ has infinite
index in $\SL_{n-1}(R')$, as claimed.

\vskip.2cm
$\bullet$ {\it Second step.}  We claim that the representation $\pi$ is irreducible.
In view of Theorem~\ref{Theo-Induced}, it suffices to show that $\Comm_G(H)=H.$

 Let $g_0\in G\setminus H$.  By the first step, $g_0Hg_0^{-1}\cap L(R)$ has infinite index in $L(R)$; hence,
  $g_0Hg_0^{-1}\cap H$  has infinite index in $H$, since $L(R)$ is a subgroup of $H.$

\vskip.2cm
$\bullet$ {\it Third step.} We claim that $\pi(C^*(G))$  contains a non-zero compact operator.

Indeed,  since $\KK$ is not  an algebraic extension over its prime field,
$R$ contains a subring $R'$ which is a copy $\ZZ$ or a copy of the polynomial ring $\FF_p[T]$, by Proposition~\ref{Pro-AmenableSLn}. The corresponding subgroup 
$$L:=L(R')$$
 of $G$ is  isomorphic to $\SL_{n-1}(\ZZ)$ or $\SL_{n-1}(\FF_p[T])$. 
Observe that $L$ is a lattice in the locally group $\GG=SL_{n-1}(\RRR)$ or 
 $\GG=SL_{n-1}(\FF_p((T^{-1})))$, 
where $\FF_p((T^{-1}))$ is the local field of Laurent series over $\FF_p$.
Since $n-1\geq 3,$  the group $\GG$ and hence $L$  has Kazhdan's Property (T); see \cite[\S. 1.4, 1.7]{BHV}.

Let $S$ be a system of representatives for the double cosets space $H\backslash G/H$
with $e\in S.$ 
 By Proposition~\ref{Prop-Induced}, the restriction  $\pi|_{L}$ to $L$ of $\pi$  is equivalent to the direct sum
 $\sigma|_L\oplus \pi_0,$ where 
 $$\pi_0:=\bigoplus_{s\in S\setminus\{e\}} \Ind_{s^{-1}H s\cap L}^L (\sigma^s|_{s^{-1}H s\cap L}).$$
 We claim that  $\pi_0$ does not weakly contain $\sigma|_L.$
Indeed, assume by contradiction that $\pi_0$  weakly contains $\sigma|_L.$
Since $\sigma$ is finite dimensional and $L$ has Property (T),  it follows that 
$\pi_0$ contains   $\sigma|_L$ (see \cite[Theorem 1.2.5]{BHV}).
Therefore, $\Ind_{s^{-1}H s\cap L}^L  (\sigma^s|_{s^{-1}H s\cap L})$ contains a  subrepresentation
of $\sigma|_L$ for some $s\in S\setminus\{e\}.$ 
Hence,  $s^{-1}H s\cap L$ has finite index in $L$, by Proposition~\ref{Prop-Induced-SubRep}.
Since $L=L(R')$ for an infinite unital subring $R'$ of $R,$
this is a contradiction to the first step.

As in the proof for the case $n=3,$ we conclude that $\pi(C^*(G))$  contains a non-zero compact operator.

This proves Theorem~\ref{Theorem2} for $G= \GL_n(R)$ when $n\geq 3.$
 The case $G=\SL_n(R)$ is proved in exactly the same way.
 
 \subsection{Proof of Corollary~\ref{Cor-Theorem2}}
 \label{SS-Proofs-Cor-Theo2}

 For $n\geq 3,$ let  $G=GL_n(R)$ for a ring $R$ as above.
 The irreducible  traceable representations of $G$ constructed in the proof of  Theorem~\ref{Theorem2}
 are of the form $\pi= \Ind_H ^G \sigma$ for a  finite dimensional  representation of 
 the subgroup   $H= \begin{pmatrix} R^\times & R^{n-1} \\ 0 & GL_{n-1}(R) \end{pmatrix}$.
 Observe that, $\pi= \Ind_H ^G \sigma$ is trivial on the center $Z$ of $G,$ since $H$ contains $Z.$

By Theorem~\ref{Theo-Induced},   there are infinitely (respectively, uncountably)  many non equivalent such representations $\pi$, provided there exists infinitely (respectively, uncountably)  non equivalent  finite dimensional irreducible representations of $H.$
This will be the case if $GL_{n-1}(R) \ltimes R^{n-1},$ which  is a quotient of $H$,  has infinitely (respectively, uncountably)  many non equivalent   finite dimensional  irreducible representations.

\medskip
\noindent 
 (i) 
 Assume that $R=\KK$ is an infinite field. 
It is easy to show that the finite dimensional irreducible  representations of $GL_{n-1}(\KK) \ltimes \KK^{n-1}$ 
are all of the form 
$$
 \begin{pmatrix} * & *\\ 0 & A \end{pmatrix} \mapsto \chi(\det A), \qquad A\in \GL_{n-1}(\KK)
 $$
 for some $\chi$ in the unitary dual $\widehat{\KK^*}$ of $\KK^*$;
 as $\KK^*$ is infinite, $\widehat{\KK^*}$ is a compact infinite group and is therefore uncountable. 
 
\medskip
\noindent 
 (ii)  Assume that  $R=\ZZ$. 
 
  \noindent
$\bullet$ {\it Case $n=3.$}  The  free group $F_2$
is a subgroup of finite index in  $\GL_2(\ZZ)$. There exists uncountably many 
 unitary characters  (that is one-dimensional unitary representations)  of $F_2$.
For every such unitary character   $\chi$, the representation $\Ind_{F_2}^{\GL_2(\ZZ)}\chi$ 
is finite dimensional and so has a decomposition  $\oplus_{i} \sigma_i^{(\chi)}$
 as a direct sum of  finite  dimensional   irreducible  representations $ \sigma_i^{(\chi)}$ of $\GL_2(\ZZ)$. 
One can choose uncountably many   pairwise non equivalent   
representations among the $\sigma_i^{(\chi)}$'s and  we obtain in this way
uncountably many non equivalent  finite dimensional irreducible representations of $\GL_2(\ZZ)$
and hence of $\GL_{2}(\ZZ)\ltimes \ZZ^{2}.$

\noindent
$\bullet$ {\it Case $n\geq 4.$}  The   group $\GL_{n-1}(\ZZ)\ltimes \ZZ^{n-1}$
 has Kazhdan's property (T)  and so has at most  countably many 
 non equivalent finite dimensional  representations (see \cite[Theorem 2.1]{Wang}).
 There are indeed \emph{infinitely}  many  such representations: 
  for every  integer $N\geq 1,$ the finite group  
  $$G_N=\GL_{n-1}(\ZZ/N\ZZ))\ltimes (\ZZ/N\ZZ)^{n-1}$$
  is a quotient of $\GL_{n-1}(\ZZ)\ltimes \ZZ^{n-1}$;
infinitely many  representations among the irreducible representations of the $G_N$'s
 are  pairwise non equivalent  when viewed as representations of  $\GL_{n-1}(\ZZ)\ltimes \ZZ^{n-1}$.


\begin{thebibliography}{BeMa00}

\bibitem[Bek--07]{Bekk--07}
B.\ Bekka.
Operator-algebraic superrigidity for $\SL_n(\ZZ)$, $n \ge 3$.
Invent.\ Math.\ \textbf{169} (2007), 401--425.

\bibitem[Bek--16]{Bekk--16}
B.\ Bekka. 
Spectral rigidity of group actions on homogeneous spaces.
Prepint 2017,  to appear in  ``Handbook of group actions, Volume III" (Editors: L. Ji, A. Papadopoulos, 
S-T Yau);  ArXiv 1602.02892.

\bibitem[BHV--08]{BHV} B. Bekka, P. de la Harpe, A. Valette. \newblock \emph{Kazhdan's Property (T).}
Cambridge University Press 2008.

\bibitem[CaM--84]{CaMo--84}
A.\ Carey and W.\ Moran.
{Characters of nilpotent groups}.
Math. \ Proc. Camb. Phil. Soc.\ \textbf{96} (1984), no.\ 1, 123--137.


\bibitem[Cor--75]{Corw--75} 
L.\ Corwin.
{Induced representations of discrete groups}.
Proc.\ Amer.\ Math.\ Soc.\ \textbf{47} (1975), 279--287. 

\bibitem[Cun--83]{Cuntz} J.\ Cuntz.
$K$-theoretic amenability for discrete groups. J. Reine Angew. Math. 344 (1983), 180--195


\bibitem[Dix--77]{Dixm--C*}
J.\ Dixmier.
\emph{$C^*$-algebras}.
 North-Holland 1977.
 



 \bibitem[Gli--61]{Glim--61}
J.\ Glimm.
{Type I C*-algebras}.
Ann.\ of Math.\ (2) \textbf{73} (1961), 572--612.
 
 
 \bibitem[Guic--63]{Guic--63}
A.\ Guichardet.
{Caract\`eres des alg\`ebres de Banach involutives}.
Ann.\ Inst.\ Fourier \textbf{13} (1963), 1--81.

\bibitem[Guiv--80]{GuivAsym} Y. Guivarc'h.
Quelques propri\'et\'es asymptotiques des produits de matrices al\'eatoires.
\newblock Lecture Notes  Math. { 774}, 17-250,  
\newblock Springer, 1980.


\bibitem[How--77]{Howe--77}
R.\ Howe.
{On representations of discrete, finitely generated,
torsion-free, nilpotent groups}.
Pacific J.\ Math.\ \textbf{73} (1977), no.\ 2, 281--305.


\bibitem[HoR--89]{HoRo--89}
R.\ Howe, J.\ Rosenberg.
{The unitary representation theory of $\GL(n)$ of an infinite discrete field}.
Israel J.\ Math.\ \textbf{67} (1989), 67--81.

\bibitem[Jul--15]{Julg--15}  P.\ Julg. 
A new look at the proof of $K$-theoretic amenability for groups acting on trees. 
Bull. Belg. Math. Soc. Simon Stevin \textbf{22} (2015),  263--269. 

\bibitem[JuV--84]{JuVa--84} P.\ Julg,  A.\ Valette.
$K$-theoretic amenability for ${\rm SL}\sb{2}({\bf Q}\sb{p})$, and the action on the associated tree. 
J. Funct. Anal. \textbf{58} (1984),194--215.



\bibitem[Kiri--65]{Kirillov}
A.A.\ Kirillov.
{Positive definite
functions on a group of matrices with elements from a discrete field.}
Soviet.\ Math.\ Dokl.\ \textbf{ 6} (1965), 707--709.

\bibitem[Kle--61]{Klep--61}
A.\ Kleppner.
{On the intertwining number theorem}.
Proc.\ Amer.\ Math.\ Soc.\ \textbf{12} (1961), 731--733. 




\bibitem[Mac--51]{Mack--51}
G.W.\ Mackey.
{On induced representations of groups}.
Amer.\ J.\ Math.\ \textbf{73} (1951), 576--592.

\bibitem[Mac--52]{Mack--52}
G.W.\ Mackey.
{Induced representations of locally compact groups. I}.
Ann.\ of Math.\ (2) \textbf{55} (1952), 101--139.



\bibitem[PeT--16]{PeTh--16}
J. Peterson, A. Thom.
{Character rigidity for special linear groups.}
J.\ reine angew.\ Math.\ \textbf{716} (2016), 207--228. 


\bibitem[Ros--89]{Rose--89} J.\ Rosenberg.
{Un compl\'ement \`a un th\'eor\`eme de Kirillov
sur les caract\`eres de $\GL(n)$ d'un corps infini discret.}
C.R.\ Acad.\ Sci.\ Paris \textbf{309} (1989), S\'erie I, 581--586.

\bibitem[Ser--80]{Serr--80}  J-P. \ Serre.
\emph{Trees.}  Springer-Verlag, Berlin-New York, 1980.



\bibitem[Szw--91]{Szwa--91}
R.\ Szwarc. 
Groups acting on trees and approximation properties of the Fourier algebra. 
J. Funct. Anal. \textbf{95} (1991), 320--343.


\bibitem[VeK--91]{Vershik} A.M. Vershik, S.V. Kerov.
Asymptotic theory of the characters of a symmetric group. 
Functional Anal. Appl. \textbf{15} (1982), 246--255.



\bibitem[Tho--68]{Thom--68}
E.\ Thoma.
{Eine Charakterisierung diskreter Gruppen vom Typ I}.
Invent.\ Math.\ \textbf{6} (1968), 190--196.

 
\bibitem[Wan--75]{Wang} S.P.\ Wang.
On isolated points in the dual spaces of locally compact groups. 
Math. Ann.\ \textbf{218} (1975),  19--34. 


\bibitem[Wei--67]{Weil} A. Weil.
\emph{Basic number theory. }
Die Grundlehren der mathematischen Wissenschaften \textbf{144}, Springer 1967.
\end{thebibliography}
\end{document}